\documentclass{article}
\usepackage[margin=1.4in]{geometry} 

\usepackage{amssymb,amsfonts,amsmath,bbm,mathrsfs,stmaryrd, mathtools, tcolorbox}
\usepackage{xcolor}
\usepackage{url}
\usepackage{dsfont}
\usepackage{enumerate}
\usepackage{multicol}
\usepackage{enumitem}
\usepackage{tikz-cd}
\usetikzlibrary{cd}

\usepackage[colorlinks,
             linkcolor=black!75!red,
             citecolor=blue,
             pdftitle={},
             pdfauthor={Mitch Hamidi, Lara Ismert, Brent Nelson},
             pdfproducer={pdfLaTeX},
             pdfpagemode=None,
             bookmarksopen=true
             bookmarksnumbered=true]{hyperref}

\usepackage{tikz}
\usetikzlibrary{cd,arrows,calc,decorations.pathreplacing,decorations.markings,intersections,shapes.geometric,through,fit,shapes.symbols,positioning,decorations.pathmorphing}

\usepackage{braket}

\usepackage[amsmath,thmmarks,hyperref]{ntheorem}
\usepackage{cleveref}

\creflabelformat{enumi}{#2(#1)#3}

\crefname{section}{Section}{Sections}
\crefformat{section}{#2Section~#1#3} 
\Crefformat{section}{#2Section~#1#3} 

\crefname{subsection}{\S}{\S\S}
\crefformat{subsection}{#2\S#1#3} 
\Crefformat{subsection}{#2\S#1#3} 

\theoremstyle{plain}

\newtheorem{lemma}{Lemma}[section]
\newtheorem{proposition}[lemma]{Proposition}
\newtheorem{corollary}[lemma]{Corollary}
\newtheorem{theorem}[lemma]{Theorem}

\theoremstyle{nonumberplain}

\theoremstyle{plain}
\theorembodyfont{\upshape}
\theoremsymbol{\ensuremath{\blacklozenge}}

\newtheorem{definition}[lemma]{Definition}
\newtheorem{example}[lemma]{Example}
\newtheorem{remark}[lemma]{Remark}

\crefname{definition}{definition}{definitions}
\crefformat{definition}{#2definition~#1#3} 
\Crefformat{definition}{#2Definition~#1#3} 

\crefname{ex}{example}{examples}
\crefformat{example}{#2example~#1#3} 
\Crefformat{example}{#2Example~#1#3} 

\crefname{remark}{remark}{remarks}
\crefformat{remark}{#2remark~#1#3} 
\Crefformat{remark}{#2Remark~#1#3} 

\crefname{convention}{convention}{conventions}
\crefformat{convention}{#2convention~#1#3} 
\Crefformat{convention}{#2Convention~#1#3} 

\crefname{exercise}{exercise}{exercises}
\crefformat{exercise}{#2exercise~#1#3} 
\Crefformat{exercise}{#2Exercise~#1#3}

\crefname{lemma}{lemma}{lemmas}
\crefformat{lemma}{#2lemma~#1#3} 
\Crefformat{lemma}{#2Lemma~#1#3} 

\crefname{proposition}{proposition}{propositions}
\crefformat{proposition}{#2proposition~#1#3} 
\Crefformat{proposition}{#2Proposition~#1#3} 

\crefname{corollary}{corollary}{corollaries}
\crefformat{corollary}{#2corollary~#1#3} 
\Crefformat{corollary}{#2Corollary~#1#3} 

\crefname{theorem}{theorem}{theorems}
\crefformat{theorem}{#2theorem~#1#3} 
\Crefformat{theorem}{#2Theorem~#1#3} 

\crefname{assumption}{assumption}{Assumptions}
\crefformat{assumption}{#2assumption~#1#3} 
\Crefformat{assumption}{#2Assumption~#1#3} 

\crefname{equation}{}{}
\crefformat{equation}{(#2#1#3)} 
\Crefformat{equation}{(#2#1#3)}

\theoremstyle{nonumberplain}
\theoremsymbol{\ensuremath{\blacksquare}}

\newtheorem{proof}{Proof.}

\newcommand{\Span}[1]{\text{Span}\{ #1\}}

\newcommand\bC{{\mathbb C}}

\newcommand\bN{{\mathbb N}}

\newcommand\cC{{\mathcal C}}
\newcommand\cD{{\mathcal D}}

\newcommand\cF{{\mathcal F}}
\newcommand\cG{{\mathcal G}}

\newcommand\cJ{{\mathcal J}}
\newcommand\cK{{\mathcal K}}
\newcommand\cL{{\mathcal L}}

\newcommand\cO{{\mathcal O}}

\newcommand\cT{{\mathcal T}}

\newcommand{\Tr}{\text{Tr}} 

\newcommand{\C}{\mathbb{C}} 
\newcommand{\N}{\mathbb{N}} 
\newcommand{\T}{\mathbb{T}} 

\newcommand{\ep}{\varepsilon} 

\newcommand{\<}{\left\langle} 
\renewcommand{\>}{\right\rangle}
\newcommand{\lp}{\left(}
\newcommand{\rp}{\right)}

\newcommand{\lan}{\left\langle}
\newcommand{\ran}{\right\rangle}

\newcommand{\norm}[1]{\left\lVert#1\right\rVert} 
\newcommand{\spn}[1]{\text{Span}\lp #1\rp} 

\newcommand{\abs}[1]{\left\lvert #1 \right\rvert} 

\newcommand{\id}{\text{id}}

\title{Simplicity of Cuntz--Pimsner algebras of quantum graphs}
\author{Mitch Hamidi
            \footnote{Department of Mathematics, Embry-Riddle Aeronautical University \hfill \url{hamidim@erau.edu}}
            \and Lara Ismert
             \footnote{Department of Mathematics, University of Nebraska-Lincoln \hfill \url{lara.ismert@gmail.com}}
             \and Brent Nelson
             \footnote{Department of Mathematics, Michigan State University \hfill \url{brent@math.msu.edu}}
             }

\begin{document}

\date{}

\maketitle

\begin{abstract}
Let $\cG$ be a quantum graph without quantum sources and $E_\cG$ be the quantum edge correspondence for $\cG.$ Our main results include sufficient conditions for simplicity of the Cuntz--Pimsner algebra $\cO_{E_\cG}$ in terms of $\cG$ and for defining a surjection from the quantum Cuntz--Krieger algebra $\cO(\cG)$ onto a particular relative Cuntz--Pimsner algebra for $E_\cG$. As an application of these two results, we give the first example of a quantum graph with distinct quantum Cuntz--Krieger and local quantum Cuntz--Krieger algebras.  

We also characterize simplicity of $\cO_{E_\cG}$ for some fundamental examples of quantum graphs, including quantum graphs on a single full matrix algebra, complete quantum graphs, and trivial quantum graphs. Along the way, we provide an equivalent condition for minimality of $E_\cG$ and sufficient conditions for aperiodicity of $E_\cG$ in terms of the underlying quantum graph $\cG$.

\end{abstract}

\section*{Introduction}
A classical simple, finite graph $G$ can be described as a pair $(V,A)$, where $V=\{1,\dots, d\}$ is a finite set, called the vertices of $G$, and $A$ is a $d\times d$ $\{0,1\}$-matrix  called the {\em adjacency matrix} for $G$. Equivalently, $A:\C^d\to \C^d$ is a linear map which is Schur idempotent. Introduced in \cite{CK80}, the Cuntz--Krieger algebra $\cO_A$ is the universal C*-algebra generated by a Cuntz-Krieger $G$-family, a collection of partial isometries $\{S_i:1\leq i\leq d\}$ with mutually orthogonal range projections which satisfy the {\em Cuntz--Krieger relation}:
\[
\forall 1\leq i\leq d: \quad S_i^*S_i= \sum_{j=1}^d A_{i,j}S_jS_j^*.
\]
Cuntz--Krieger algebras have been studied extensively and generalize Cuntz algebras, introduced in \cite{Cun77}, and are useful in studying symbolic dynamics endowed by $A$ on the infinite path space of $G$.

A {\em finite quantum graph} $\cG$ is a triple $(B,\psi,A)$ consisting of a {\em quantum set} $(B,\psi)$, where $B$ is a finite-dimensional C*-algebra and $\psi$ is a state on $B$, and a {\em quantum adjacency matrix} $A$, which is a linear map on $B$ satisfying a condition called {\em quantum Schur idempotence}. Any classical simple, finite graph $G=(V,A)$ can be realized as a finite quantum graph via Gelfand duality. We quantize $G$ to a quantum set $(\C^{\abs{V}},\psi)$, where $\psi:\C^{\abs{V}}\to \C$ is the state corresponding to the uniform counting measure on $V$, together with $A$ viewed as an endomorphism on $\C^{\abs{V}}$. 

It turns out that classical Schur idempotence of the adjacency matrix $A$ is equivalent to quantum Schur idempotence of $A$ viewed as an endomorphism on the quantum set $(\C^{\abs{V}},\psi)$, and thus, $(\C^{\abs{V}}, \psi, A)$ is a finite quantum graph. Hence, a finite quantum graph $(B,\psi,A)$ with noncommutative $B$ is a noncommutative analogue of a classical finite graph.

The notion of a quantum adjacency matrix for a quantum set was first introduced by Musto-Reutter-Verdon \cite{MRV19} and further generalized by Brannan et. al. \cite{BCE20}, and it endows a quantum set with a structure that is equivalent to the notion of a {\em quantum relation} on a quantum set, introduced by N.~Weaver in \cite{Weaver10}. 
With a quantum analogue of a $\{0,1\}$-matrix in hand, Brannan et. al. introduced a C*-algebra arising from a quantum graph $\cG=(B,\psi,A)$ in \cite{BEVW22}, called the {\em (free) quantum Cuntz--Krieger (QCK) algebra} for $\cG$. In \cite{BHINW23}, the present paper's authors, along with M. Brannan and M. Wasilewski, added an extra relation to their definition to ensure unitality of the generated C*-algebra, and we call this C*-algebra the {\em quantum Cuntz--Krieger (QCK) algebra} for $\cG$, denoted $\cO(\cG)$. 

Though we will recall a rigorous definition of $\cO(\cG)$ in \Cref{section:Preliminaries}, roughly speaking, we build $\cO(\cG)$ from a QCK $\cG$-family, which is a collection of matrix partial isometries $\{S^{(a)}:1\leq a\leq d\}$ indexed by the matrix blocks of $B\cong \oplus_{a=1}^d M_{n(a)}$ that satisfy an appropriate analogue of the Cuntz--Krieger relation. The QCK algebra $\cO(\cG)$ is then generated by the {\em coefficients} of the matrix blocks in the QCK $\cG$-family. However, because only the matrix blocks, not their coefficients, are assumed to be partial isometries, the relations tell us very little about the coefficients themselves. Similarly, the quantum Cuntz--Krieger relations on the matrix blocks give rise to relations on the coefficient generators which are coarser than the classical Cuntz--Krieger relations. Consequently, the isomorphism class of $\cO(\cG)$ is difficult to ascertain in general (though some special cases associated complete quantum graphs have been resolved, see \cite[Theorem 4.5]{BEVW22}). 

For these reasons, a particular quotient of $\cO(\cG)$ by an ideal generated by so-called ``local" relations was introduced in \cite{BHINW23}. We call this quotient the {\em local quantum Cuntz--Krieger (LQCK) algebra} for $\cG$, denoted $L\cO(\cG)$. In most cases, these additional relations are robust enough to identify $L\cO(\cG)$ with a certain Cuntz--Pimsner algebra $\cO_{E_\cG}$, which we describe in more detail below. Moreover, in these cases it remained a possibility that the quantum Cuntz--Krieger algebra $\cO(\cG)$ was isomorphic to its local counterpart. In \Cref{ex:main_result}, we provide a concrete collection of examples where this fails. 

The primary tool we develop to separate $\cO(\cG)$ from $\cO_{E_\cG}$ in \Cref{ex:main_result} is developed in \Cref{sec:toeplitz_reps}, wherein we consider a certain relative Cuntz--Pimsner algebra $\cO(K,E_\cG)$ which canonically surjects onto $\cO_{E_\cG}$, and we show in \Cref{thm:qck-g-family_iff_covariant_on_range} that $\cO(\cG)$ surjects onto $\cO(K,E_\cG)$. In particular, when this specific ideal $K$ is non-trivial, or equivalently, the C*-correspondence $E_\cG$ is not full (see \cite[Theorem 2.9]{BHINW23}), we have that $\cO(K,E_\cG)$ is non-simple. As a corollary, we obtain the following main result of this paper.
\vspace{.15cm}

 \noindent {\bf \Cref{cor:QCK-ne-to-QCP}.}
    Let $\cG$ be a quantum graph with no quantum sources and a nonempty set of quantum sinks. If $\cO_{E_\cG}$ is simple, then the surjection of $\cO(\cG)$ onto $\cO_{E_\cG}$ is not injective.
 \vspace{.15cm}
 
The remainder of the paper focuses on characterizing simplicity of $\cO_{E_\cG}$. In \Cref{prop:nonreturning_vector_not_full}, we identify a condition on $\cG$ which provides a rich collection of {\em non-returning vectors} for $E_\cG$. Such vectors are needed to check if a C*-correspondence satisfies Condition (S), a condition that was introduced in \cite{ME22} and is part of a sufficient set of conditions on $E_\cG$ which guarantee simplicity of the associated Cuntz--Pimsner algebra. We then use \Cref{prop:nonreturning_vector_not_full} to show simplicity of $\cO_{E_\cG}$ in \Cref{ex:main_result}. 

The end of \Cref{section:Main_Example} includes a brief discussion about the utility of checking Condition (S) for C*-correspondences arising from quantum graphs. Studying conditions for simpicity of Cuntz--Krieger algebras became of particular interest following R\o rdam's main result in \cite{Rordam95}: the $K_0$-group is a complete invariant for stable isomorphism of simple Cuntz--Krieger algebras. In \Cref{sec:simplicity}, we leverage results from \cite{Marrero-Muhly05} to characterize simplicity of the Cuntz--Pimsner algebra for rank-one quantum graphs on a single full matrix algebra in \Cref{example:rank-one_quantum_graph}, and we apply the main result from \cite{SCHWEIZER01} to prove the Cuntz--Pimsner algebras arising from complete quantum graphs are always simple, while the Cuntz--Pimsner algebras arising from trivial quantum graphs are always non-simple. Along the way, we provide an equivalent condition for minimality of $E_\cG$ strictly in terms of the underlying quantum graph $\cG$ in \Cref{prop:minimality_quantum_graph}, and we give sufficient conditions for aperiodicity of $E_\cG$ in terms of the underlying graph $\cG$ in \Cref{prop:if_periodic_then_no_sinks_and_no_sources}.

\section*{Acknowledgments}
The authors are grateful to their colleagues C. Nelson and R. Hamidi-Ismert, whose timely arrival greatly enhanced (and greatly delayed) the preparation of this article. They also wish to thank C. Schafhauser for being a sounding board for the commutative case of the main example and M. Wasilewski for providing suggestions to improve the final draft of the paper. The second author was supported by NSF grant DMS-2247587, and the third author was supported by NSF grant DMS-2247047.


\section{Preliminaries}\label{section:Preliminaries}

In this section, we recall definitions and fundamental theorems pertaining to quantum graphs and their quantum edge correspondences, (local) quantum Cuntz--Krieger algebras, and (relative) Cuntz--Pimsner algebras.

\subsection{Quantum graphs and quantum edge correspondences}
Throughout, let $B$ be a finite-dimensional C*-algebra viewed as the direct sum of full matrix algebras:
\[
B
\cong
\bigoplus_{a=1}^d M_{n(a)}(\C).
\]
Given a state $\psi:B\to \C$, define $m:L^2(B\otimes B,\psi\otimes \psi)\to L^2(B,\psi)$ by $m(x\otimes y):=xy$, and let $m^*$ denote the Hermitian adjoint of $m$. Given $\delta>0$, we say $\psi$ is a {\em $\delta$-form} if $mm^*=\delta^2 \id_B$. Equivalently, $\psi$ is a $\delta$-form if the density matrix $\rho=\oplus_{a=1}^d \rho_a\in B$ associated to $\psi$ is invertible and satisfies $\Tr(\rho_a^{-1})=\delta^2$ for each $1\leq a\leq d$. 

A linear map $A:B\to B$ is {\em quantum Schur idempotent} if it satisfies
\[
m(A\otimes A)m^*
=
\delta^2 A.
\]
In this paper, we will also require $A$ to be completely positive (cp), and then we shall refer to $A$ as a {\em quantum adjacency matrix} for $(B,\psi)$. A finite quantum graph is a triple $\cG=(B,\psi,A)$.  

Define $\ep_\cG$ to be $\delta^{-2}(\id_B \otimes A)m^*(1_B)$. In \cite{BHINW23}, we prove a variety of facts about the properties of $\ep_\cG$ and its interplay with $A$. We define the quantum edge correspondence $E_\cG$ to be the $B$-$B$-subcorrespondence $\Span{x\cdot \ep_\cG\cdot y:x,y \in B}$ of $B\otimes_\psi B$ whose left and right $B$-actions and $B$-valued inner product are inherited from the usual left and right $B$-actions and $B$-valued inner product on $B\otimes_\psi B$:
\begin{itemize}
    \item $a\cdot (x\otimes y):=(ax)\otimes y$\vspace{-.25cm}
    \item $(x\otimes y)\cdot b:=x\otimes (yb)$ \vspace{-.25cm}
    \item $\langle x_1\cdot \ep_\cG\cdot y_1\,|\,x_2\cdot \ep_\cG\cdot y_2\rangle
    =\delta^{-2}y_1^*A(x_1^*x_2)y_2.$ 
\end{itemize}


\begin{theorem}[{{\cite[Theorem 2.9]{BHINW23}}}]\label{thm:faithful_full_correspondence}
Let $\cG=(B,\psi,A)$ be a quantum graph such that $\psi$ is a $\delta$-form and $A$ is cp, and let $E_\cG$ be the quantum edge correspondence. Then
    \[
        \{x\in B\colon x\cdot \xi=0\ \forall \xi\in E_\cG\}= \left(BA^{\ast}(B)B\right)^{\perp}
    \]
and
    \[
        \overline{\text{span}}\<E_\cG\,|\,E_\cG\> = 
        B\cdot A(B) \cdot B,
    \]
    the two-sided ideal of $B$ generated by the range of $A$. 
In particular, $E_\cG$ is faithful if and only if  $\ker(A)$ does not contain a central summand of $B$, and $E_\cG$ is full if and only if $A(B)$ is not orthogonal to a central summand of $B$.
\end{theorem}

We call a central summand of $B$ which is contained in $\ker(A)$ a {\em quantum source}, and we call a central summand of $B$ which is orthogonal to the range of $A$ a {\em quantum sink}.

\subsection{Quantum Cuntz--Krieger algebras}

    \begin{definition}[\cite{BEVW22},\cite{BHINW23}]\label{def:QCK}
Let $\cG=(B,\psi,A)$ be a finite quantum graph such that $\psi$ is a $\delta$-form and $A$ is cp. A \textit{quantum Cuntz--Krieger $\cG$-family} in a unital C*-algebra $\cD$ is a linear map $s\colon B\to \cD$ such that:
    \begin{enumerate}[label=(\roman*)]
        \item $\mu_\cD(\mu_\cD\otimes 1)(s\otimes s^* \otimes s)(m^* \otimes 1)m^*=s$ $\hfill (\textbf{QCK1})$
    
        \item $\mu_\cD(s^*\otimes s)m^* = \mu_\cD(s\otimes s^*)m^*A$ $\hfill (\textbf{QCK2})$
        
        \item $\mu_\cD (s \otimes s^*)m^*(1_B) = \frac{1}{\delta^2} 1_\cD$ $\hfill (\textbf{QCK3})$
        
    \end{enumerate}
where $\mu_\cD\colon \cD\otimes \cD\to \cD$ is the multiplication map for $\cD$ and  $s^*(b)=s(b^*)^*$ for $b\in B$. A QCK $\cG$-family $S:B\to \cD$ is {\em universal} if for any other QCK $\cG$-family $s:B\to \cC$, there exists a $*$-homomorphism $\rho:\cD\to \cC$ such that $s=\rho\circ S.$ The \textit{quantum Cuntz--Krieger algebra} associated to $\cG$ is the unital C*-algebra  $\mathcal{O}(\cG)$ generated by a universal quantum Cuntz--Krieger $\cG$-family.
\end{definition}


\begin{definition}[{{\cite[Definition 3.4]{BHINW23}}}]
Let $\cG=(B,\psi,A)$ be a finite quantum graph such that $\psi$ is a $\delta$-form. A \textit{local quantum Cuntz--Krieger $\cG$-family} in a unital C*-algebra $\cD$ is a linear map $s\colon B\to \cD$ such that
    \begin{enumerate}[label=(\roman*)]
        \item $\mu_\cD(\mu_\cD \otimes 1)(s \otimes s^* \otimes s)(m^* \otimes 1) = \frac{1}{\delta^2}s m$ $\hfill(\textbf{LQCK1})$
        
        \item $\mu_\cD (s^* \otimes s) = \frac{1}{\delta^2} \mu_\cD(s \otimes s^*)m^*Am$ $\hfill(\textbf{LQCK2})$
        
        \item $\mu_\cD (s \otimes s^*)m^*(1_B) = \frac{1}{\delta^2} 1_\cD$ $\hfill (\textbf{LQCK3})$
        
    \end{enumerate}
The {\em local quantum Cuntz--Krieger algebra associated to $\cG$} is the C*-algebra $L\cO(\cG)$ generated by the image of a universal local quantum Cuntz--Krieger $\cG$-family.
\end{definition}


\begin{theorem}[{{\cite[Theorem 3.7]{BHINW23}}}]\label{thm:qck_quotients_onto_lqck} Let $\cG$ be a finite quantum graph such that $\psi$ is a $\delta$-form and $A$ is cp, and let $E_\cG$ be its quantum edge correspondence. Assume that  $\ker(A)$ does not contain a central summand of $B$. Then $\cO_{E_\cG}\cong \cO(\cG)/\cJ$, where $\cJ$ is the closed two-sided ideal generated by the relations (LQCK1), (LQCK2), and (LQCK3).
 \end{theorem}


\subsection{Relative Cuntz--Pimsner algebras}
\Cref{section:Preliminaries} culminates in \Cref{prop:separation_of_rel_CP_and_CP}, which plays a vital role in our main result \Cref{cor:QCK-ne-to-QCP}. While \Cref{prop:separation_of_rel_CP_and_CP} is unsurprising (and may even be well known to experts), it appears to contradict some of the previous literature (see Remark~\ref{remark:oops}), and so the authors felt it was important to formally prove this proposition.

Let $X$ be a C*-correspondence over a C*-algebra $B$ with right $B$-valued inner product $\langle \cdot | \cdot \rangle: X\times X\to B$ and left action induced by a $*$-homomorphism $\varphi_X:B\to \cL(X)$. A \emph{Toeplitz representation} of $X$ in a C*-algebra $\cD$ is a pair $(\pi,t)$ consisting of a $*$-homomorphism $\pi\colon B\to \cD$ and a linear map $t\colon X\to \cD$ satisfying:\vspace{-.15cm}
    \begin{enumerate}[label=(\roman*)]
        \item $\pi(b)t(\xi)=t(\varphi_X(b) \xi)$ for $b\in B$ and $\xi\in X$,\vspace{-.15cm}
        \item $t(\xi)^*t(\eta)=\pi(\<\xi\,|\,\eta\>)$ for $\xi,\eta\in X$. \vspace{-.15cm}
    \end{enumerate}
Using (i), one can also show $t(\xi)\pi(b)= t(\xi\cdot b)$ for $b\in B$ and $\xi\in X$. The {\em Toeplitz algebra for $X$}, denoted $\cT_X$, is the C*-algebra generated by the universal Toeplitz representation $(\overline{\pi}_X,\overline{t}_X)$. That is, if $(\pi,t)$ is a Toeplitz representation of $X$ in a C*-algebra $\cD$, there exists a $*$-homomorphism $\rho:\cT_X\to D$ such that $\pi=\rho\circ \overline{\pi}_X$ and $t=\rho\circ \overline{t}_X$. A nice feature of the Toeplitz algebra for a C*-correspondence is that its universal representation can be defined concretely on the Fock module $\cF(X)$. See \cite[Proposition 6.5]{KATSURA04} for details. 

In the present paper, we will always consider $(\overline{\pi}_X,\overline{t}_X)$ to be the canonical representation of $X$ as adjointable operators on $\cF(X)$: given $n\geq 1$ and $\xi=\xi_1\otimes \dots \otimes \xi_n\in X^{\otimes n}$, \vspace{-.15cm}
\begin{itemize}
    \item for $b\in B$, define $\overline{\pi}_X(b)\xi:=\varphi_X(b)\xi_1\otimes \xi_2\otimes \dots \otimes \xi_n$ and $\overline{\pi}_X(b)b':=bb'$ for $b'\in B.$ \vspace{-.15cm}
    \item  for $\eta\in X$, define $\overline{t}_X(\eta)\xi:=\eta\otimes \xi$ and $\overline{t}_X(\eta)b':=\eta\cdot b'$ for $b'\in B.$ \vspace{-.15cm}
\end{itemize}

For $\xi,\eta\in X$ define the rank-one operator $\theta_{\xi,\eta}$ on $X$ by $\theta_{\xi,\eta}(\mu):=\xi\cdot \langle \eta\,|\,\mu\rangle$ for all  $ \mu \in X.$ Let $\cK(X)$ denote the norm-closed span of all rank-one operators in the C*-algebra $\cL(X)$ of adjointable operators on $X$. A Toeplitz representation $(\pi,t)$ of $X$ in a C*-algebra $\cD$ gives rise to a canonical $*$-homomorphism \vspace{-.15cm}
\[
\psi_t\colon \cK(X)\to \cD\quad  \text{ by }\quad\psi_t(\theta_{\xi,\eta})=t(\xi)t(\eta)^*,\quad \forall \xi,\eta\in X. \vspace{-.15cm}
\]

Note in some parts of the literature, $\psi_t$ is instead denoted by $\pi^{(1)}$. Introduced in \cite{KATSURA04}, the \emph{Katsura ideal} is $J_X:=\varphi_X^{-1}(\cK(X))\cap (\ker \varphi_X)^\perp$, where $(\ker \varphi_X)^\perp=\{y\in B:xy=0 \text{ for all }y\in \ker{\varphi_X}\}.$

Given a subset $S\subseteq B,$ a Toeplitz representation $(\pi,t)$ of $X$ is {\em co-isometric on $S$} if $\pi(b)=\psi_t(\varphi_X(b))$ for all $b\in S$. When $\pi(b) = \psi_t(\varphi_X(b))$ for all $b\in J_X$, the Toeplitz representation $(\pi,t)$ is called {\em covariant}. Let $I\trianglelefteq B$ be a closed two-sided $*$-ideal (henceforth simply referred to as {\em an ideal}) that is contained in $J_X$. Introduced in \cite{Muhly-Solel98} and further fleshed out in \cite{FMR2003} and \cite{KATSURA07}, the \emph{relative Cuntz--Pimsner algebra} for $X$ relative to $I$, denoted $\cO(I,X)$, is the C*-algebra generated by a representation $(\pi_X^I,t_X^I)$ which is universal with respect to Toeplitz representations that are co-isometric on $I$. That is, given any Toeplitz representation $(\pi,t)$ of $X$ in a C*-algebra $\cD$ which is co-isometric on $I$, there exists a $*$-homomorphism $\rho:  \cO(I,X)\to \cD$ satisfying $\pi=\rho\circ \pi_X^I$ and $t=\rho\circ t_X^I$, i.e., the following diagram commutes:

 
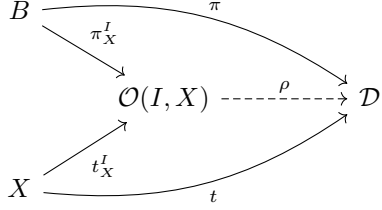
\begin{center}
       \begin{tikzpicture} \node[scale=1] (a) {
\begin{tikzcd}[ampersand replacement=\&]
B\arrow{dr}{\pi_X^I}\arrow[bend left=20]{drrr}{\pi} \& \& \& \\
\& \cO(I,X) {\arrow[dashed]{rr}{\rho}}\& \& \cD \\
X \arrow[swap]{ur}{t_X^I}\arrow[bend right=20, swap]{urrr}{t}  \&  \& \&\\
 \end{tikzcd}};
 \end{tikzpicture}
 \end{center}
 \vspace{-.5cm}
 
 The \emph{Cuntz--Pimsner algebra} $\cO_X$ is defined to be the relative Cuntz--Pimsner algebra $\cO(J_X,X)$, and one can show the Toeplitz algebra $\cT_X$ is the relative Cuntz--Pimsner algebra $\cO(\{0\},X)$.  Rather than denoting the universal Toeplitz covariant representation by $(\pi_X^{J_X}, t_X^{J_X})$, we simply write $(\pi_X,t_X)$. Similarly, instead of $(\pi^{\{0\}}_X,t^{\{0\}}_X)$, we write $(\overline{\pi}_X,\overline{t}_X)$ for the universal Toeplitz representation.

For the sake of self-containment, the next two results needed to justify \Cref{prop:separation_of_rel_CP_and_CP} are included. Some proofs are outlined, while we provide only references for others.


\begin{lemma}\label{lem:canonical_embeddings_Fock_module}
   Let $X$ be a $B$-correspondence and $I$ an ideal in $B$ which is contained in $J_X$. \vspace{-.15cm}
   \begin{enumerate}
     \item $\cK(\cF(X)I)=\overline{\spn{\{\theta_{\xi\cdot b, \eta}:\xi\in X^{\otimes n},\,\eta\in X^{\otimes m}\text{ for }m,n\geq 0,\; b\in I\}}}$. \vspace{-.15cm}
       \item There is a canonical embedding $\iota^I:\cK(\cF(X)I)\hookrightarrow \cK(\cF(X))$ such that $\iota^I(\theta^I_{\xi,\eta})=\theta_{\xi,\eta}$ for each $\xi,\eta\in \cF(X)I$, where $\theta^I_{\xi,\eta}$ is the usual rank-one operator on the $B$-correspondence $\cF(X)I.$ \vspace{-.15cm}
       \item $\iota^I(\cK(\cF(X)I))$ is the ideal in $\cT_X$ in generated by $\{\overline{\pi}_X(b)-\psi_{\overline{t}}(\varphi_X(b)):b\in I\}$. \vspace{-.15cm}
   \end{enumerate}  
\end{lemma}

\begin{proof}
Statement (1) can be deduced from \cite[Lemma 1.3]{KATSURA07}. Statement (2) follows by \cite[Lemma 2.6]{FMR2003} and the proof \cite[Proposition 3.14]{FMR2003}. 

For Statement (3), first recall $\cK(\cF(X)J_X)\subseteq C^*(\overline{\pi}_X, \overline{t}_X)$ by \cite[Proposition 4.6]{KATSURA04}. An arbitrary element of $\cK(\cF(X)I)$ can be approximated by a linear combination of rank-one operators of the form $\theta^I_{\xi\cdot b,\eta}$ for some $m,n\geq 0,$ $\xi\in X^{\otimes n},\eta\in X^{\otimes m}$, and $b\in I$. One can verify that as an operator on $\cF(X)$, \vspace{-.15cm}
\[
\iota^I(\theta^I_{\xi\cdot b,\eta})
=
\theta_{\xi\cdot b,\eta}
=\overline{t}_X^n(\xi)[\overline{\pi}_X(b)-\psi_{\overline{t}}(\varphi_X(b))]\overline{t}_X^m(\eta)^*, \quad \quad \quad (\star) \vspace{-.15cm}
\]
where for each $k\in \bN$, the map $\overline{t}_X^k:X^{\otimes k}\to \cL(\cF(X))$ is defined inductively on $\xi_1\otimes \dots \otimes \xi_k\in X^{\otimes k}$ by \vspace{-.15cm}
\[\overline{t}_X^k(\xi_1\otimes \dots \otimes \xi_k):=\overline{t}_X(\xi_1)\overline{t}^{k-1}_X(\xi_2\otimes \dots \otimes \xi_k).\vspace{-.15cm}\]
Therefore, $\iota^I(\cK(\cF(X)I))$ is contained in the ideal of $\cT_X$ generated by $\{\overline{\pi}_X(b)-\psi_{\overline{t}}(\varphi_X(b)): b\in I\}$. 

 For the opposite containment, let $P_0$ denote the projection of $\cF(X)$ onto the constant summand $B$, and set $\overline{\pi}_0(b):=\overline{\pi}_X(b)P_0$ for each $b\in B.$ By \cite[Proposition 4.4]{KATSURA04}, the difference $\overline{\pi}_X(b)-\psi_{\overline{t}}(\varphi_X(b))$ is equal to $\overline{\pi}_0(b)$ for all $b\in J_X$. Because $P_0 \overline{\pi}_X(b)=\overline{\pi}_X(b)P_0$ for all $b\in B$ while $t(\eta)^*P_0=P_0t(\xi)=0$ for any $\xi,\eta\in X$, the ideal generated by $\{\overline{\pi}_X(b)-\psi_{\overline{t}}(\varphi_X(b)):b\in I\}$ inside $\cT_X$ is the closed span of elements which have the form of that on the right side of $(\star)$. Statement (3) follows.
\end{proof}

The proof of Statement (3) above identifies the canonical embedding of $\cK(\cF(X)I)$ inside $\cK(\cF(X))$ with the ideal in $\cT_X$ generated by $\overline{\pi}_0(I)=\{\overline{\pi}_0(b):b\in I\}$. This will be useful later.


\begin{proposition}\label{prop:rel_CK_tool}
    Let $I$ be an ideal in $B$ which is contained in $J_X$. \vspace{-.15cm}
    
    \begin{enumerate}
        \item Denote the quotient map $\cT_X\twoheadrightarrow \cT_X/\iota^I(\cK(\cF(X)I))$ by $\overline{\tau}_I$. Then $\pi_X^I=\overline{\tau}_I\circ \overline{\pi}_X$ and $t_X^I=\overline{\tau}_I\circ \overline{t}_X,$ i.e., the following diagram commutes:
 \begin{center}
       \begin{tikzpicture} \node[scale=1] (a) {
\begin{tikzcd}[ampersand replacement=\&]
B\arrow{dr}{\overline{\pi}_X}\arrow[bend left=20]{drrr}{\pi_X^I} \& \& \& \\
\& \cT_X {\arrow{rr}{\overline{\tau}_I}}\& \& \cO(I,X) \\
X \arrow[swap]{ur}{\overline{t}_X}\arrow[bend right=20, swap]{urrr}{t_X^I}  \&  \& \&\\
 \end{tikzcd}};
 \end{tikzpicture}
 \end{center}
 \vspace{-.5cm}
 In particular, $\cO(I,X)$ is isomorphic to $\overline{\tau}_I(\cT_X)$, and $\ker(\overline{\tau}_I)$ is isomorphic to $\iota^I(\cK(\cF(X)I)).$
         \item 
   Denote the quotient map $\cO(I,X)\twoheadrightarrow \cO(I,X)/\overline{\tau}_I(\iota^I(\cK(\cF(X)J_X)))$ by $\tau^I$. The following diagram commutes:
   
 \begin{center}
       \begin{tikzpicture} \node[scale=1] (a) {
\begin{tikzcd}[ampersand replacement=\&]
B\arrow{dr}{\pi_X^I}\arrow[bend left=20]{drrr}{\pi_X} \& \& \& \\
\& \cO(I,X) {\arrow{rr}{\tau^I}}\& \& \cO_X \\
X \arrow[swap]{ur}{t_X^I}\arrow[bend right=20, swap]{urrr}{t_X}  \&  \& \&\\
 \end{tikzcd}};
 \end{tikzpicture}
 \end{center}
\vspace{-.5cm} 
 In particular, the Cuntz--Pimsner algebra $\cO_X$ is isomorphic to the image of $\tau^I$, and $\ker(\tau^I)$ is isomorphic to $\overline{\tau}_I(\iota^I(\cK(\cF(X)I))).$
     \end{enumerate}
    
Consequently, the quotient map $\overline{\tau}_{J_X}:\cT_X\twoheadrightarrow \cT_X/\iota^{J_X}(\cK(\cF(X)J_X))$ is equal to the composition $\tau^I\circ \overline{\tau}_I$ for any ideal $I$ of $B$ which is contained in $J_X.$
\end{proposition}
\begin{proof}
    Statement (1) is verified in the proof of \cite[Proposition 1.3]{FMR2003} wherein the authors show that $(\overline{\tau}_I\circ \overline{\pi}_X,\overline{\tau}_I\circ \overline{t}_X)$ has the universal property defining $(\pi_X^I,t_X^I)$. Hence, $\cO(I,X)$ can be concretely constructed as the quotient of $\cT_X=C^*(\overline{\pi}_X,\overline{t}_X)$ by the ideal $\iota^I(\cK(\cF(X)I)$, which is generated by $\overline{\pi}_0(I)$ by \Cref{lem:canonical_embeddings_Fock_module}.    Statement (2) and the statement $\overline{\tau}_{J_X}=\tau^I\circ \overline{\tau}_I$ are justified in the proof of  \cite[Proposition 3.14]{FMR2003} prior to invoking the proposition's hypothesis that $X$ is full.
\end{proof}


Our strategy for constructing an explicit example of a quantum graph whose quantum and local quantum Cuntz--Krieger algebras are non-isomorphic relies on the fact that whenever $J_X=B$ and $I$ is a proper ideal of $B$, the canonical surjection $\tau^I:\cO(I,X)\twoheadrightarrow \cO_X$ is not injective. When $X$ is full, \cite[Proposition 3.14]{FMR2003} states that the kernel of $\tau^I$ is Morita equivalent to $J_X/I$. Hence, under these hypotheses, we know $\cO(I,X)$ and $\cO_X$ are non-isomorphic whenever $\cO_X$ is simple. However, fullness of the correspondence is too restrictive for our purposes; our technique for proving $\cO(\cG)$ and $\cO_{E_\cG}$ are non-isomorphic requires that the ideal $\langle E_\cG\,|\,E_\cG\rangle$ is proper in $B$. 


\begin{proposition}\label{prop:separation_of_rel_CP_and_CP}
    Suppose $B$ is finite-dimensional, $X$ is a $B$-correspondence, and $\varphi_X$ is faithful. If $I\triangleleft B$ is a proper ideal, then the canonical surjection $\tau^I:\cO(I,X)\twoheadrightarrow \cO_X$ is not injective. 
\end{proposition}

\begin{proof}
    Under these hypotheses, the Katsura ideal for $X$ is all of $B$. Given a proper ideal $I$ of $B$, there exists a nonzero central projection $q$ in $B$ which does not lie in $I$. Set $Q:=\overline{\tau}_I(\overline{\pi}_0(q))\in \cO(I,X)$. We will show $Q\neq 0$ and $\tau^I(Q)=0$. 
    
    Towards a contradiction, suppose $\overline{\pi}_0(q)$ were in the kernel of $\overline{\tau}_I$. By \Cref{prop:rel_CK_tool}, this means  $\overline{\pi}_0(q)$ is in the ideal generated by $\overline{\pi}_0(I)$. By \Cref{lem:canonical_embeddings_Fock_module} (and the discussion following it) and because $\overline{\pi}_0(I)$ is finite-dimensional, we have that $\overline{\pi}_0(q)$ is a linear combination of elements of the form $\overline{t}^n_X(\xi)\overline{\pi}_0(b)\overline{t}^m_X(\eta)^*$ for some $n,m\geq 0$, $\xi\in X^{\otimes m}$, $\eta\in X^{\otimes n}$, and $b\in I$. However, $\overline{\pi}_0(q)$ is nonzero and commutes with $P_0$, while $\overline{t}^m_X(\eta)^*P_0=0$ and $P_0\overline{t}^n_X(\xi)=0$ whenever $m>0$ and $n>0$, respectively. Hence, since $\overline{\pi}_0$ is a homomorphism, $\overline{\pi}_0(q)=\overline{\pi}_0(b)$ for some $b\in I$. Thus, given $1_B\in \cF(X)$, we have $\overline{\pi}_0(q)1_B=q$ while $\overline{\pi}_0(b)1_B=b$, which implies $q=b\in I$, a contradiction. Therefore, $\overline{\pi}_0(q)\not\in \ker(\overline{\tau}_I)$.

    Applying \Cref{prop:rel_CK_tool} and the universal property of $\cO_X$, the composition $\tau^I\circ \overline{\tau}_I:\cT_X\twoheadrightarrow \cO_X$ is precisely the quotient map $\overline{\tau}_{J_X}:\cT_X\twoheadrightarrow \cT_X/\iota^{J_X}(\cK(\cF(X)J_X))$. When $J_X=B$, we have $\ker(\overline{\tau}_{J_X})\cong \overline{\pi}_0(B)$ by \Cref{lem:canonical_embeddings_Fock_module}. Therefore, $\tau^I(Q)=\overline{\tau}_{J_X}(\overline{\pi}_0(q))=0.$ We conclude that when $I$ is a proper ideal of $B$, the canonical surjection $\tau^I:\cO(I,X)\twoheadrightarrow \cO_X$ is not injective.
\end{proof}


If $\cG$ is a quantum graph such that $E_\cG$ is faithful but not full ($\cG$ has no quantum sinks but at least one quantum source), then $J_{E_\cG}=B$ and $\langle E_\cG\;|\;E_\cG\rangle\subsetneq B$, so in particular, when $\cO_{E_\cG}$ is simple, the canonical surjection of the relative Cuntz--Pimsner algebra $\cO(\langle E_\cG\,|\,E_\cG\rangle, E_\cG)$ onto $\cO_{E_\cG}$ is not injective by \Cref{prop:separation_of_rel_CP_and_CP}, and therefore, $\cO(\langle E_\cG\,|\,E_\cG\rangle, E_\cG)$ is not isomorphic to $\cO_{E_\cG}$. 


\begin{remark}\label{remark:oops}
\Cref{prop:separation_of_rel_CP_and_CP} describes C*-correspondences that are counterexamples to \cite[Remark 3.15]{FMR2003}, which states that in the case when $X$ is not full, $\ker(\tau^I)$ is Morita equivalent to the quotient of $J_X\cap \langle X\,|\,X\rangle$ by $I\cap \langle X\,|\,X\rangle$. This remark implies, in particular, that if $X$ is not full, $J_X=B,$ and $I=\langle X\,|\,X\rangle,$ then the relative Cuntz--Pimsner algebra $\cO(\langle X\,|\,X\rangle,X)$ is isomorphic to $ \cO_{X}$. We believe \cite[Remark 3.15]{FMR2003} is likely correct outside the case when $\langle X\,|\, X\rangle\subseteq J_X$.
\end{remark}


\section[QCK G-families from Toeplitz representations of the quantum edge correspondence]{QCK $\cG$-families from Toeplitz $E_\cG$-representations}\label{sec:toeplitz_reps}
Throughout this section, let $\cG=(B,\psi,A)$ be a finite quantum graph such that $\psi$ is a $\delta$-form and $A$ is cp, and let $E_\cG$ denote the quantum edge correspondence for $\cG$ cyclically generated as a $B$-bimodule by $\ep_\cG:=\delta^{-2}(\id\otimes A)m^*(1_B)$. As $B$ is assumed to be finite-dimensional and $E_\cG$ is cyclically generated as a $B$-$B$-bimodule, $E_\cG$ is finite-dimensional. Thus, $\mathcal{L}(E_\cG)=\cK(E_\mathcal{G}).$ Let $\varphi_{E_\cG}:B\to \mathcal{L}(E_\cG)$ denote the left action of $B$ on $E_\cG$ as adjointable (compact) operators. We sometimes denote $\varphi_{E_\cG}(x)(\xi)$ by $x\cdot \xi$ for $x\in B$ and $\xi\in E_\cG$. We record some results from \cite{BHINW23}. 


\begin{theorem}[{{\cite[Theorem 3.6]{BHINW23}}}]\label{thm:universal_local_quantum_Cuntz--Krieger_algebra}
Let $\cG=(B,\psi,A)$ be a quantum graph such that $\psi$ is a $\delta$-form, $A$ is cp, and  $E_\cG=B\cdot \ep_\cG\cdot B$ is faithful. Let $(\pi_{E_\cG},t_{E_\cG})$ denote the universal covariant representation of $E_\cG$ on the Cuntz--Pimsner algebra $\cO_{E_\cG}$.
Then $S\colon B\to \mathcal{O}_{E_\cG}$ defined by $S(x):=\frac{1}{\delta} t_{E_{\cG}}(x\cdot \ep_\cG)$ is a local quantum Cuntz--Krieger $\cG$-family whose image generates $\cO_{E_\cG}$. Moreover, given any local quantum Cuntz--Krieger $\cG$-family $s\colon B\to \cD$ in a unital C*-algebra $\cD$ there exists a $*$-homomorphism $\rho\colon \mathcal{O}_{E_\cG}\to \cD$ such that $s=\rho\circ S$.
\end{theorem}


As a corollary, whenever the quantum edge correspondence for a quantum graph $\cG$ is faithful ($\cG$ has no quantum sources), the Cuntz--Pimsner algebra $\cO_{E_\cG}$ is isomorphic to the local quantum Cuntz--Krieger algebra for $\cG$. 
Let $(B,\psi)$ be a finite quantum set and let $\{e_{ij}^{(a)}:1\leq a\leq d, 1\leq i,j\leq n(a)\}$ denote the standard matrix units for $B\cong \bigoplus_{a=1}^d M_{n(a)}(\C).$ For $1\leq a\leq d$ and $1\leq i,j\leq n(a)$, define 
\[
f_{ij}^{(a)}:=\psi(e_{ii}^{(a)})^{-1/2}e_{ij}^{(a)}\psi(e_{jj}^{(a)})^{-1/2},
\]
which satisfy
    \[
        m^*(f_{ij}^{(a)}) = \sum_{k=1}^{n(a)} f_{ik}^{(a)} \otimes f_{kj}^{(a)}
    \]
and $(f_{ij}^{(a)})^*=f_{ji}^{(a)}$ (see \cite[Lemma 3.2]{BEVW22}). We call the collection $\{f_{ij}^{(a)}:1\leq a\leq d, 1\leq i,j\leq n(a)\}$ {\em adapted matrix units} for $(B,\psi)$.


\begin{theorem}[{{\cite[Theorem 2.12]{BHINW23}}}]\label{thm:compact_operators}
If $\{ f_{ij}^{(a)}\colon 1\leq a\leq d,\ 1\leq i,j\leq n(a)\}$ are the adapted matrix units for $(B,\psi)$, then
    \[
        \varphi_{E_\cG}(f_{ij}^{(a)}) = \sum_{k=1}^{n(a)} \theta_{f_{ik}^{(a)}\cdot \ep_\cG,f_{jk}^{(a)}\cdot \ep_\cG}
    \]

\end{theorem}


\begin{corollary}[{{\cite[Corollary 2.13]{BHINW23}}}]\label{cor:abstract_Toeplitz}
If $(\pi,t)$ is a covariant representation of $E_\cG$ in a C*-algebra $\cD$, then the linear map $T\colon B\to \cD$ defined by $T(x):= t(x\cdot \ep_\cG)$ satisfies
\begin{equation}\label{eq:Covariant_rep_1}
        \mu_\cD(T^*\otimes T) = \frac{1}{\delta^2}\pi A m
        \end{equation}
     \vspace{-0.5cm}   \begin{equation}\label{eq:Covariant_rep_2}
        \mu_\cD(T\otimes T^*) m^* =\psi_t\varphi_{E_\cG},
 \end{equation}
where $\mu_\cD\colon \cD\otimes \cD\to \cD$ is the multiplication map and $T^*(x):=T(x^*)^*$.
\end{corollary}


First we prove a theorem similar to  \Cref{cor:abstract_Toeplitz}, but without assuming covariance of $(\pi, t)$.

\begin{theorem}\label{cor:abstract_Toeplitz_new}
If $(\pi,t)$ is a Toeplitz representation of $E_\cG$ in a C*-algebra $\cD$, then the linear map $T\colon B\to D$ defined by $T(x):= t(x\cdot \ep_\cG)$ satisfies

 \begin{equation}\label{eq:Toeplitz_rep_1}
  \mu_\cD(T^*\otimes T)m^*=\pi A
 \end{equation}   
  \vspace{-.5cm}
  \begin{equation}\label{eq:Toeplitz_rep_2}
     \mu_\cD(T\otimes T^*) m^* =\psi_t\varphi_{E_\cG}
 \end{equation}
 \vspace{-.5cm}
\begin{equation}  \label{eq:Toeplitz_rep_3}  \mu_\cD(\psi_t\varphi_{E_\cG}\otimes T)m^* = \delta^2 T
\end{equation}  
where $\mu_\cD\colon \cD\otimes \cD\to \cD$ is the multiplication map, $T^*(x):=T(x^*)^*$, and $\psi_t\colon \cK(E_\cG)\to \cD$ is the $*$-homomorphism induced by $t$.
\end{theorem}

\begin{proof}
Fix an adapted matrix unit $f_{ij}^{(a)}\in B.$
\begin{enumerate}
    \item Observe
    \begin{align*}
        \mu_\cD(T\otimes T^*)m^*(f_{ij}^{(a)})
        &=\mu_\cD\left(\sum_{k=1}^{n(a)} T^*(f_{ik}^{(a)})\otimes T(f_{kj}^{(a)}) \right)\\
        &=\mu_\cD\left(\sum_{k=1}^{n(a)} t(f_{ki}^{(a)}\cdot \ep_{\mathcal{G}})^*\otimes t(f_{kj}^{(a)}\cdot \ep_{\mathcal{G}}) \right)\\
        &=\sum_{k=1}^{n(a)} \pi\left(\braket{ f_{ki}^{(a)}\cdot \ep_\mathcal{G}\,|\,f_{kj}^{(a)}\cdot \ep_{\mathcal{G}}}\right)\\
        &=\sum_{k=1}^{n(a)} \delta^{-2}\pi(A(f_{ik}^{(a)}f_{kj}^{(a)}))\\
        &=\sum_{k=1}^{n(a)} \delta^{-2}\psi(e_{kk}^{(a)})^{-1}\pi(A(f_{ij}^{(a)}))\\
        &=\pi(A(f_{ij}^{(a)}))
    \end{align*}
   By linearity, we conclude $\mu_\cD(T\otimes T^*)m^*=\pi A$.
   
    \item Observe
    \begin{align*}
        \mu_\cD(T\otimes T^*)m^*(f_{ij}^{(a)})
        &= \mu_\cD\left(\sum_{k=1}^{n(a)} T(f_{ik}^{(a)})\otimes T^*(f_{kj}^{(a)})\right)\\
        &= \sum_{k=1}^{n(a)}t(f_{ik}^{(a)}\cdot \ep_{\mathcal{G}})t(f_{jk}^{(a)}\cdot \ep_\mathcal{G})^*\\
         &= \sum_{k=1}^{n(a)}\psi_t\left(\theta_{f_{ik}^{(a)}\cdot \ep_{\mathcal{G}},f_{jk}^{(a)}\cdot \ep_\mathcal{G}}\right)\\
         &=\psi_t(\varphi_{E_\cG}(f_{ij}^{(a)}))\quad \text{by Theorem \ref{thm:compact_operators}}
    \end{align*}
    By linearity, $\mu_\cD(T\otimes T^*)m^*=\psi_t\varphi_{E_\cG}$.
    
    \item Observe
    \begin{align*}
        \mu_\cD(\psi_t\varphi_{E_\cG}\otimes T)m^*(f_{ij}^{(a)})
        &=\sum_{k=1}^{n(a)} \psi_t\varphi_{E_\cG}(f_{ik}^{(a)}) t(f_{kj}^{(a)}\cdot \ep_\mathcal{G})\\
        &=\sum_{k=1}^{n(a)} t(f_{ik}^{(a)}f_{kj}^{(a)}\cdot \ep_{\mathcal{G}})\quad \text{ by \cite[Lemma 2.4]{KATSURA04}}\\
        &=\sum_{k=1}^{n(a)} \psi(e_{kk}^{(a)})^{-1} t(f_{ij}^{(a)}\cdot \ep_{\mathcal{G}})\\
        &=\delta^2 T(f_{ij}^{(a)}).
    \end{align*}
    By linearity, $\mu_\cD(\psi_t\varphi_{E_\cG}\otimes T)m^* = \delta^2 T.$
\end{enumerate}
\end{proof}


\begin{theorem}\label{thm:qck-g-family_iff_covariant_on_range}

If $(\pi,t)$ is a Toeplitz representation of $E_\cG$ in a C*-algebra $\cD$ and $T(x):=t(x\cdot \ep_\cG)$, then the map
 $S:=\frac{1}{\delta} T$ satisfies {\bf QCK 1} and {\bf QCK 3}. Moreover, the map
 $S$ satisfies {\bf QCK 2} if and only if $\pi(A(x))=\psi_t(\varphi_{E_\cG}(A(x)))$ for all $x\in B$, that is, if and only if $(\pi,t)$ is co-isometric on the range of $A$.
\end{theorem}
\begin{proof}
{\bf (QCK 1)} Observe
\begin{align*}
    \mu_\cD(\mu_\cD\otimes \text{id})(S\otimes S^*\otimes S)(m^*\otimes \text{id})m^*
    &=
    \frac{1}{\delta^3}
    \mu_\cD(\mu_\cD\otimes \text{id})(T\otimes T^*\otimes T)(m^*\otimes \text{id})m^*\\
    &=\frac{1}{\delta^3}
    \mu_\cD[\mu_\cD(T\otimes T^*)m^*\otimes T]m^*\\
    &=\frac{1}{\delta^3}
    \mu_\cD[\psi_t\otimes T]m^*  \quad \Cref{eq:Toeplitz_rep_2}\\
    &=\frac{1}{\delta^3}(\delta^2T) \quad \Cref{eq:Toeplitz_rep_3} \\
    &=S.
\end{align*}
{\bf (QCK 3)} Note $\varphi_{E_\cG}(1_B)\in K(E_\mathcal{G})$ and $\psi_t(\varphi_{E_\cG}(1_B))=\psi_t(\id_{E_\cG})=1_\cD$. Hence, 
\[
    \mu_\cD(S\otimes S^*)m^*(1_B)
    =
    \frac{1}{\delta^2}
     \mu_\cD(T\otimes T^*)m^*(1_B)
    =\frac{1}{\delta^2}
    \psi_t(\varphi_{E_\cG}(1_B)) 
    =\frac{1}{\delta^2}1_\cD. 
\]
Note that $S$ satisfying {\bf QCK 2} means
$$
    \mu_\cD(T^*\otimes T)m^*= \mu_\cD(T\otimes T^*)m^*A.
$$
By \Cref{eq:Toeplitz_rep_1} and \Cref{eq:Toeplitz_rep_2}, this is equivalent to $\psi_t\varphi_{E_\cG} A = \pi A.$
\end{proof}


Suppose $\cG=(B,\psi,A)$ has quantum sinks, i.e., the ideal $(B\cdot A(B)\cdot B)^\perp$ is non-empty. Let $K$ be the ideal generated by the range of $A$, denoted $K=B\cdot A(B)\cdot B$. Recall from \Cref{thm:faithful_full_correspondence} that $K$ is precisely $\langle E_\cG\,|\, E_\cG\rangle,$ and thus is nontrivial if and only if $E_\cG$ is not full.


\begin{proposition}\label{prop:duh}
    If $(\pi,t)$ is a Toeplitz representation of $E_\cG$ in a unital C*-algebra $\cD$ which is co-isometric on $K:=B\cdot A(B)\cdot B$, then $S:B\to D$ defined by $S(x):=\frac{1}{\delta}t(x\cdot \ep_\cG)$ is a QCK $\cG$-family and $C^*(S)=C^*(\pi,t)$.
\end{proposition}
\begin{proof}
\Cref{thm:qck-g-family_iff_covariant_on_range} guarantees $S$ is a QCK $\cG$-family since $(\pi,t)$ is co-isometric on $K=B\cdot A(B)\cdot B$. To show $C^*(S)=C^*(\pi,t)$, we need only show $C^*(S)\supseteq C^*(\pi,t)$. Recall that $C^*(\pi,t)$ is the closure of the $*$-algebra generated by the image of $t$ (\cite[Proposition 2.7]{KATSURA04}). Thus, we will show that $t(x\cdot \ep\cdot y)\in C^*(S)$ for all $x,y\in B$ in two cases: $y\in K$ and $y\in K^\perp,$ where $K^\perp$ is the ideal generated by the central projections of $B$ which are not contained in $K$. By linearity of $t$ and $S$, it will follow that $C^*(S)=C^*(\pi,t)$. 

Suppose $f_{ij}^{(a)}\in K$. Note $t(\ep_\cG\cdot f_{ij}^{(a)})=t(\ep_\cG)\pi(f_{ij}^{(a)})=\delta S(1)\pi(f_{ij}^{(a)}).$ Thus, it suffices to prove $\pi(f_{ij}^{(a)})\in C^*(S)$. Since $(\pi,t)$ is co-isometric on $K$, \Cref{thm:compact_operators} yields
$$
\pi(f_{ij}^{(a)})
=\sum_{k=1}^{n(a)} \psi_t\left(\theta_{f_{ik}^{(a)}\cdot \ep_\cG, f_{jk}^{(a)}\cdot \ep_\cG}\right)
=\sum_{k=1}^{n(a)} t(f_{ik}^{(a)}\cdot \ep_\cG) t(f_{jk}^{(a)}\cdot \ep_\cG)^*
=\delta^2\sum_{k=1}^{n(a)}S_{ik}^{(a)}(S_{jk}^{(a)})^*\in C^*(S).
$$
By linearity of $\pi$, $\pi(y)\in C^*(S)$ for any $y\in K,$ and therefore $t(x\cdot \ep_\cG \cdot y)=\delta S(x)\pi(y)\in C^*(S)$ for any $x\in B$ and $y\in K.$ If $y\in K^\perp$, then $\delta^2\lan \ep_\cG\cdot y\,|\,\ep_\cG\cdot y\ran=y^*A(1)y=0$. Therefore, $\ep_\cG\cdot y=0$. Hence, for any $x\in B$, $t(x\cdot \ep_\cG\cdot y)=0\in C^*(S)$. 
\end{proof}


\begin{corollary}\label{cor:surjection_onto_rel_CP_alg}
Let $K:=B\cdot A(B)\cdot B$. There is a surjection from the QCK algebra $\cO(\cG)$ onto the relative Cuntz--Pimsner algebra $\cO(K,E_\cG).$
\end{corollary}
\begin{proof}
    \Cref{prop:duh} implies that the universal representation 
    $(\pi^K_{E_\cG}, t^K_{E_\cG})$ of $\cO(K,E_\cG)$ gives rise to a QCK $\cG$-family. By the universal property of $\cO(\cG)$, there is a surjection from $\cO(\cG)$ onto $\cO(K,E_\cG)$.
\end{proof}


\begin{corollary}\label{cor:QCK-ne-to-QCP}
    Let $\cG$ be a quantum graph with no quantum sources and a nonempty set of quantum sinks. If $\cO_{E_\cG}$ is simple, then $\cO(\cG)$ and $\cO_{E_\cG}$ are non-isomorphic.
\end{corollary}
\begin{proof}
    By \Cref{cor:surjection_onto_rel_CP_alg}, $\cO(\cG)$ surjects onto $\cO(K,E_\cG)$, which is non-simple by \Cref{prop:separation_of_rel_CP_and_CP}. It follows that $\cO(\cG)$ is non-simple and hence cannot be isomorphic to $\cO_{E_{\cG}}$. 
\end{proof}

  
\section{Non-returning vectors and Condition (S)}\label{section:Main_Example}

Given a quantum graph $\cG:=(B,\psi,A)$ , set $K:=B\cdot A(B)\cdot B$. In order to construct an example of a quantum graph $\cG$ such that $\cO_{E_\cG}$ is simple and distinct from the relative Cuntz--Pimsner algebra $\cO(K,E_\cG)$, we identify a class of {\em non-returning vectors} in $E_{\cG}$ when $E_\cG$ is not full. Recall that, given a C*-correspondence $X$ over a C*-algebra $B$, we denote the universal Toeplitz covariant representation of $X$ into $\cO_X$ by $(\pi, t)$. Introduced in \cite{ME22}, a vector $\xi\in X^{\otimes m}$ for some $b\in \bN$ is \textit{non-returning} if whenever $0< n<m$ and $\eta\in X^{\otimes n}$, then \[t^m(\xi)^*t^n(\eta)t^m(\xi)=0.\] The following lemma shows how we can use orthogonal central projections in the coefficient algebra of a C*-correspondence to construct non-returning vectors.


\begin{lemma}\label{lemma:nrv-by-projs}
Suppose $p,q$ are elements in a C*-algebra $B$ such that $q^*p=0$ and $q$ is central.  If $X$ is a C*-correspondence over $B$ and $\{\xi_i\}_{i=1}^m$ are vectors in $X$, then \[(p\cdot \xi_1\cdot q)\otimes (\xi_2\cdot q)\otimes (\xi_3\cdot q)\otimes \cdots \otimes (\xi_m \cdot q)\in X^{\otimes m}\] is a non-returning vector. 

\begin{proof}
Let $(\pi, t)$ be the universal covariant representation of $X$ into $\cO_X$ and fix $0<n < m$. Set $\xi=(p\cdot\xi_1\cdot q)\otimes (\xi_2\cdot q)\otimes \cdots  \otimes (\xi_{n} \cdot q)\in X^{\otimes n}$ and $\xi'=(\xi_{n+1}\cdot q)\otimes \cdots \otimes (\xi_m \cdot q)\in X^{\otimes m-n}$. Since $p,q$ are central projections in $B$ and $\xi=\xi\cdot q$, for all $\eta\in X^{\otimes n}$ we have 
\begin{align*}
t^{n}(\xi)^*t^n(\eta)t(p\cdot\xi_1\cdot q)&= \pi(\<\xi,\eta\>)t(p\cdot\xi_1\cdot q)\\
&=t(\<\xi,\eta\>p\cdot \xi_1\cdot q)\\
&=t(q\<\xi, \eta\>p\cdot \xi_1\cdot q)\\
&=0.
\end{align*}

Hence, we have $t^m(\xi\otimes \xi')^*t^n(\eta)t^m(\xi\otimes\xi')=0,$ and so $\xi\otimes\xi'$ is a non-returning vector.
\end{proof}
\end{lemma}

\Cref{prop:nonreturning_vector_general} is an immediate corollary of \Cref{lemma:nrv-by-projs} and yields a family of non-returning vectors for quantum edge correspondences. 

\begin{proposition}\label{prop:nonreturning_vector_general}
If $1_a$ and $1_b$ are orthogonal central projections in $B$, then \[(1_a\cdot \ep_\cG \cdot 1_b)\otimes (\ep_\cG \cdot 1_b) \otimes (\ep_\cG\cdot 1_b) \otimes \cdots \otimes (\ep_\cG\cdot 1_b) \in E_\cG^{\otimes m}\] is a non-returning vector.
\end{proposition}


\begin{proposition}\label{prop:nonreturning_vector_not_full}
 If $1_c$ is a central projection in $B$ orthogonal to the range of $A$, \[(1_c\cdot \ep_\cG) \otimes \ep_\cG  \otimes \ep_\cG \otimes \cdots \otimes \ep_\cG \in E_\cG^{\otimes m}\] is a non-returning vector.
\begin{proof}
Let $(\pi, t)$ be the universal covariant representation of $E_\cG$ into $\cO_{E_\cG}$. Observe that for all $x\cdot \ep_\cG\cdot y\in E_\cG$, 
\[\pi(\<\ep_\cG\,|\,x\cdot \ep_\cG\cdot y\>)t(1_c\cdot \ep_\cG) = \frac{1}{\delta^2}t(A(x)y1_c\cdot \ep_\cG)=0\] since $1_c$ is central and orthogonal to the range of $A$. It follows from a straightforward computation that $(1_c\cdot \ep_\cG) \otimes \ep_\cG  \otimes \ep_\cG \otimes \cdots \otimes \ep_\cG$ is non-returning.
\end{proof}
\end{proposition}


It is worth noting that the two prototypes of non-returning vectors introduced in  \Cref{prop:nonreturning_vector_general} and \Cref{prop:nonreturning_vector_not_full} could be zero without the hypothesis that the kernel of $A$ contains no central summands of $B$, or equivalently, $E_\cG$ is faithful. We assume that $E_\cG$ is faithful in the following section when we invoke \Cref{prop:nonreturning_vector_not_full}.

Note also that \Cref{prop:nonreturning_vector_not_full} demonstrates when the ideal $B\cdot A(B)\cdot B$ is non-trivial, or equivalently, $E_\cG$ is not full (\Cref{thm:faithful_full_correspondence}), there exist even more non-returning vectors for $E_\cG$. This is good news when appealing to Condition (S) to check simplicity of $E_\cG$, because $E_\cG$ is not full, we cannot appeal to \Cref{thm:Schweizer}, a main result of \cite{SCHWEIZER01}. For us, this is the reason to use Condition (S) to characterize simplicity $\cO_{E_\cG}$ for non-full $E_\cG$. On the other hand, checking that $E_\cG$ satisfies Condition (S) when $E_\cG$ is full has proven to be more challenging due to the lack of identified non-returning vectors as in \Cref{prop:nonreturning_vector_not_full}.


\begin{definition}[\cite{ME22}]
    A C*-correspondence $X$ over a C*-algebra $B$ satisfies \textit{Condition (S)} if whenever $x$ is a positive element in $B$, for all $n \in \N$ and $\ep>0$ there exists $m > n$ and a unit, non-returning vector $\xi\in X^{\otimes m}$ such that \[\norm{x} < \norm{\<\xi\,|\,x\cdot \xi\>} + \ep.\] 
\end{definition}

\vspace{-.25cm}

\subsection[A quantum graph whose local and quantum Cuntz--Krieger algebras are non-isomorphic]{A quantum graph \(\cG\) such that \(\cO(\cG)\) is not isomorphic to \(\cO_{E_\cG}\)}

Let $(B,\psi)$ be the finite quantum space with $B=\bigoplus_{a=1}^3 M_n(\bC)$ and $\psi$ the unique tracial $\delta$-form on $B$. Define $A:B\to B$ by \[A(x_1\oplus x_2\oplus x_3)=(x_1+x_2+x_3)\oplus (x_1+x_2)\oplus 0.\] One can verify that $A$ is a quantum adjacency matrix using the adapted matrix units in $B$ and that $A$ is cp since $A$ is the sum and composition of completely positive maps. Thus, $\cG=(B,\psi, A)$ is a directed quantum graph such that $A$ is cp.

Let $1_a$ denote the central projection onto the $a^{\text{th}}$ summand of $B$. It is clear that $\cG$ has no quantum sources since $\ker(A)$ does not contain a central summand of $B$, but the third summand of $B$ is a quantum sink in $\cG$ since $1_3$ is orthogonal to the range of $A$. Hence, $E_\cG$ is faithful but not full by Theorem 2.9 in \cite{BHINW23}. We show that $\cO_{E_\cG}$ is not isomorphic to the quantum Cuntz--Krieger algebra $\cO(\cG)$ by proving $\cO_{E_\cG}$ is simple and appealing to \Cref{cor:QCK-ne-to-QCP}. We do so by proving $E_\cG$ satisfies Condition (S) and $B$ has no non-trivial saturated, hereditary ideals as defined in \cite{ME22}. 
    Given a C*-correspondence $X$ over a C*-algebra $B$ and $I$ an ideal in $B$, define 
    \[IX:=\overline{\text{span}}\{r\cdot \xi: r \in I \text{ and } \xi \in X\}\quad \text{ and }\quad XI:=\overline{\text{span}}\{\xi\cdot r: r \in I \text{ and } \xi \in X\}.\] The ideal $I$ is \textit{hereditary} if $IX\subseteq XI$ and is {\em saturated} if whenever $b\in J_X$ and $b\cdot X\subseteq XI$, then $b \in I$.


\begin{theorem}\label{thm:condition-s-ex}
    Let $\cG=(B,\psi, A)$ be the directed quantum graph with $B=\bigoplus_{a=1}^3 M_n(\C)$, $\psi$ the unique tracial $\delta$-form on $B$, and $A:B\to B$ defined by \[A(x_1\oplus x_2\oplus x_3)=(x_1+x_2+x_3)\oplus (x_1+x_2)\oplus 0.\] The quantum edge correspondence $E_\cG$ satisfies Condition (S).

\begin{proof}
Fix $k\in \N$ and let $x \in B$ be given. Let $1_a$ and $1_b$ be orthogonal central projections in $B$ and set \[\xi:=(1_a\cdot \ep_\cG \cdot 1_b)\otimes (\ep_\cG \cdot 1_b) \otimes (\ep_\cG\cdot 1_b) \otimes \cdots \otimes (\ep_\cG\cdot 1_b) \in E_\cG^{\otimes (k+1)}.\] Define $F:B\to B$ by $F(y)=A(y)1_b$ and observe 
\[\<1_a\cdot\ep\cdot 1_b\,|\, x\cdot (1_a\cdot \ep\cdot 1_b)\>=\delta^{-2}A(x 1_a)1_b=\delta^{-2}F(x 1_a).\] Repeating a similar computation iteratively, one can show that \begin{equation}
    \<\xi\,|\, x\cdot\xi\>=\delta^{-2(k+1)} F^{k+1}(x1_a).\label{eq:ex-inp}\end{equation}
From this, we have the following cases: \[\<\xi\,|\,x\cdot \xi\>=\begin{cases}\displaystyle\frac{1}{\delta^{2(k+1)}}(0\oplus x1_1\oplus 0), & \text{if $a=1, b=2$}\\ \displaystyle\frac{1}{\delta^{2(k+1)}}(x1_2\oplus 0 \oplus 0), & \text{if $a=2, b=1$}\\ \displaystyle\frac{1}{\delta^{2(k+1)}}(x1_3\oplus 0\oplus 0), & \text{if $a=3, b=1$} \end{cases}\]

 Since $\norm{\xi}^2=\norm{\<\xi\,|\,\xi\>}=\delta^{-2(k+1)}$ using $x=1_B$ in \cref{eq:ex-inp}, renormalizing $\xi$ yields \[\norm{\<\xi\,|\, x\cdot \xi\>}=\begin{cases}\norm{x1_1}, & \text{if $a=1, b=2$}\\ \norm{x1_2}, & \text{if $a=2, b=1$}\\ \norm{x1_3}, & \text{if $a=3, b=1$} \end{cases}\]
Since the norm of $x$ is achieved on one of its three blocks, $\xi$ can be chosen so that $\norm{\<\xi\,|\, x\cdot\xi\>}=\norm{x}$. Therefore, $E_\cG$ satisfies Condition (S).
\end{proof}
\end{theorem}


\begin{example}\label{ex:main_result}
Let $\cG=(B,\psi, A)$ be the directed quantum graph with $B=\bigoplus_{a=1}^3 M_n(\C)$, $\psi$ the unique tracial $\delta$-form on $B$, and $A:B\to B$ defined by \[A(x_1\oplus x_2\oplus x_3)=(x_1+x_2+x_3)\oplus (x_1+x_2)\oplus 0.\] Then the Cuntz--Pimsner algebra $\cO_{E_\cG}$ is simple. In particular, $\cO_{E_\cG}$ is not isomorphic to the quantum Cuntz--Krieger algebra $\cO(\cG).$
\end{example}
    \begin{proof}
        By \cite[Theorem 4.3]{ME22}, $\cO_{E_\cG}$ is simple if $E_\cG$ satisfies Condition (S) and $B$ has no non-trivial saturated, hereditary ideals. It can be computed that \[\ep_\cG=\frac{1}{n}\sum_{i,j=1}^n\left[\left((e_{ij}^{(1)}+e_{ij}^{(2)}+e_{ij}^{(3)})\otimes e_{ji}^{(1)}\right)+\left((e_{ij}^{(1)}+e_{ij}^{(2)})\otimes e_{ji}^{(2)}\right)\right].\] Since the ideals of $B$ are generated by its central projections, it can be shown that the ideals $B 1_1$, $B 1_2$, $B 1_3$, $B (1_1+1_3),$ and $B (1_2 + 1_3)$ are not hereditary, and the ideal $B(1_1 + 1_2)$ is not saturated by checking the action of the appropriate central projection on $\ep_\cG$. Hence, $\cO_{E_\cG}$ is simple by \Cref{thm:condition-s-ex}, and it follows that $\cO_{E_\cG}$ is not isomorphic to $\cO(\cG)$ by \Cref{cor:QCK-ne-to-QCP}.
    \end{proof}


\begin{remark}
    When $n=1$, \Cref{ex:main_result} still holds: $\cO(\cG)$ and $\cO_{E_\cG}$ are non-isomorphic, despite $\cG$ being a commutative graph. Indeed, in this case, $\cO_{E_\cG}$ is isomorphic to $\cO_2$ while the quantum Cuntz--Krieger algebra $\cO(\cG)$ is isomorphic to a unital extension of $\cO_2$ by the compacts. 
\end{remark}

\section[Simplicity of the Cuntz--Pimsner algebra of a quantum edge correspondence]{Simplicity of $\cO_{E_\cG}$}\label{sec:simplicity}

Let $\cG=(B,\psi,A)$ be a finite quantum graph such that $A:B\to B$ is cp and $\psi$ is a $\delta$-form. Let $E_\cG$ denote the quantum edge correspondence for $\cG$ cyclically generated by $\ep_\cG$. 


\subsection[Simplicity of the Cuntz--Pimsner algebra of a quantum edge correspondence for a single-vertex quantum graph]{Simplicity of $\cO_{E_\cG}$ when $B$ is a single full matrix algebra}\label{Marrero-Muhly_Simplicity}
In \Cref{Marrero-Muhly_Simplicity}, we use a characterization of simplicity of Cuntz--Pimsner algebras arising from a more general class of C*-correspondences than our quantum edge correspondences, but with the restriction that they are C*-correspondences over a single full matrix algebra, as in \cite{Marrero-Muhly05}: Given a completely positive map $\Phi:M_n\to M_n$, consider $M_n\otimes M_n$ with the usual left and right actions by $M_n$. Define an $M_n$-valued inner product on elementary tensors in $M_n\otimes M_n$ by $\langle x\otimes y, z\otimes w\rangle:=y^*\Phi(x^*z)w$. Extending this inner product linearly makes $M_n\otimes M_n$ into a $M_n$-correspondence denoted by $M_n\otimes_\Phi M_n$.


\begin{proposition}[{{\cite[Proposition 2.6]{BHINW23}}}]\label{cor:quantum_graph_correspondence_as_cp_correspondence}
 Let $\cG=(M_n,\psi,A)$ be a finite quantum graph with $\delta$-form $\psi$ and cp map $A$. The C*-correspondences $E_\cG$ and $M_n\otimes_A M_n$ are isomorphic as $M_n$-correspondences via the map $x\cdot \ep_\cG\cdot y\mapsto \frac{1}{\delta}(x\otimes y)$.
\end{proposition}


    Given a completely positive linear map $\Phi:M_n\to M_n$, there exists a Kraus decomposition $\{K_i\}_{i=1}^p\subset M_{n}$ such that $\Phi(x)=\sum_{i=1}^p K_i^* xK_i$ for all $x\in M_n$. Define $d(\Phi)$ to be the dimension of the complex linear span of $\{K_i:1\leq i\leq p\}$.

    While the Kraus decomposition for a cp map $\Phi$ is not unique, the dimension $d(\Phi)$ is known to be independent of choice of Kraus representation. The main theorem of \cite{Marrero-Muhly05} is the following.

\begin{theorem}[\cite{Marrero-Muhly05}]\label{thm:MM05} Let $\Phi:M_n\to M_n$ be completely positive. The Cuntz--Pimsner algebra for $M_n\otimes_\Phi M_n$ is Morita equivalent to the Cuntz algebra $\cO_{d(\Phi)}$.
\end{theorem}

\begin{corollary}\label{cor:d(A)_gives_simplicity}
The Cuntz--Pimsner algebra of $M_n\otimes_\Phi M_n$ is simple if and only if $d(\Phi)> 1.$
\end{corollary}

\begin{proof}
    Because the Cuntz--Pimsner algebra of $M_n\otimes_\Phi M_n$ and $\cO_{d(\Phi)}$ are unital, \Cref{thm:MM05} tells us that the Cuntz--Pimsner algebra of $M_n\otimes_\Phi M_n$ is simple if and only if $\cO_{d(\Phi)}$ is simple, which holds if and only if $d(\Phi)>1.$
\end{proof}


\begin{example}[Rank-one quantum graphs]\label{example:rank-one_quantum_graph}
    Fix a finite quantum set $(B,\psi)$ such that $\psi$ is a $\delta$-form corresponding to a density matrix $\rho\in B$, and let $T\in M_n(\C)$ satisfy $\Tr(\rho^{-1}T^*T)=\delta^2.$ Then $A:B\to B$ given by by $A(x)=TxT^*$ for all $x\in B$ is cp and quantum Schur idempotent. We show in \cite{BHINW23} that the Cuntz--Pimsner algebra for the quantum edge correspondence $E_T$ of this quantum graph is isomorphic to $B\otimes C(\T)$, which is not simple.
\end{example}


\begin{corollary}
    Let $\cG:=(M_n,\psi,A)$ be a single-vertex quantum graph. Then $\cO_{E_\cG}$ is non-simple if and only if $\cG$ is a rank-one quantum graph, as in  \Cref{example:rank-one_quantum_graph}.
\end{corollary}

\begin{proof}
    Let $\rho$ denote the density matrix corresponding to $\psi$. By \Cref{cor:d(A)_gives_simplicity}, $\cO_{E_\cG}$ is non-simple if and only if $d(A)=1$, which holds if and only if $A(x)=TxT^*$ for some $T\in M_n$ satisfying $\Tr(\rho^{-1}T^*T)=\delta^2$. Therefore, $\cO_{E_\cG}$ is non-simple if and only if $\cG$ is a rank-one quantum graph.
\end{proof}


\begin{example}
    Consider the complete quantum graph $K(M_n,\psi)$ whose quantum adjacency matrix $A$ is defined by $A(x)=\delta^2 \psi(x)1$ for all $x\in M_n.$ Let $E_K$ denote the quantum edge correspondence for $K(B,\psi)$. In \cite{BHINW23}, we give an explicit isomorphism of $\cO_{E_K}$ with $\cO_{\dim B},$ the Cuntz algebra on $\dim B$ generators. Another lens through which to view the relationship between $\cO_{E_K}$ and $\cO_{\dim B}$ when $B$ is a single full matrix algebra follows from \Cref{thm:MM05}.

   Denote $E_{ij}$ by the outer product $\ket{e_i}\bra{e_j}$ of standard basis vectors in $\C^n.$ If $B=M_n$, a Kraus family for $A$ is the collection of adapted matrix units $\{\rho^{1/2}\psi(E_{jj})^{-1/2}E_{ij}\}_{i,j=1}^n$, where $\rho$ is the density matrix associated to $\psi$. Indeed, 
   \[
        \sum_{i,j=1}^n 
\psi(E_{jj})^{-1}E_{ji}\rho^{1/2}x\rho^{1/2}E_{ij}
        =\sum_{j=1}^n \psi(E_{jj})^{-1}\ket{e_j}\bra{e_j}\left(\Tr(\rho^{1/2}x\rho^{1/2})\right)
        =\delta^2\psi(x)1.
    \]
    Since $\rho$ is invertible, it can be shown that the span of the given Kraus family contains the standard matrix units of $M_n$, and thus $d(A) = n^2.$ Hence, $\cO_{E_K}$ is stably isomorphic to $\cO_{n^2}$, and therefore is simple.
\end{example}


An obvious drawback of relying on \Cref{thm:MM05} to characterize simplicity of $\cO_{E_\cG}$ is that it applies only to {\em single-vertex quantum graphs}, i.e., quantum graphs where $B$ is just a single full matrix algebra $M_n$. Thus, characterizing simplicity of $\cO_{E_\cG}$ when $B$ is a direct sum of matrix algebras requires a different technique. When $E_\cG$ is full, we can appeal to \Cref{thm:Schweizer} from \cite{SCHWEIZER01}. 


\subsection[Simplicity of the Cuntz--Pimsner algebra for a full quantum edge correspondence in terms of minimality and periodicity]{Simplicity of $\cO_{E_\cG}$ in terms of minimality and periodicity of $E_\cG$}

In this subsection, we determine simplicity of the Cuntz--Pimsner algebras arising from complete quantum graphs and trivial quantum graphs using Schweizer's characterization in \cite{SCHWEIZER01}. First, let us recall \cite[Definition 3.7]{SCHWEIZER01}. A C*-correspondence $X$ over a unital C*-algebra $B$ is {\em minimal} if there are no nontrivial ideals $J\triangleleft B$ such that $\langle X\,|\,JX\rangle \subseteq J$, where
\[
\langle X\,|\,JX\rangle
=\Span{\langle x\,|\,b\cdot y\rangle: x,y\in X,b\in J},
\]
and $X$ is {\em aperiodic} if whenever $X^{\otimes n}$ is isomorphic to $B$ as correspondences, $n$ must be $0.$


\begin{theorem}[{{\cite[Theorem 3.9]{SCHWEIZER01}}}]\label{thm:Schweizer}
If $X$ is a full C*-correspondence over a unital C*-algebra $B$, then the Cuntz--Pimsner algebra $\cO_X$ is simple if and only if $X$ is minimal and aperiodic.
\end{theorem}


\begin{lemma} Let $\cG:=(B,\psi,A)$ be a finite quantum graph such that $\psi$ is a $\delta$-form and $A$ is cp. For any ideal $J\trianglelefteq B$, we have $
\langle E_\cG\,|\,JE_\cG\rangle=B\cdot A(J)\cdot B.$
\begin{proof}
 Let $b\in J$ and let $x,y,z,w\in B$, and observe
\[
  \lan x\cdot \ep_\cG \cdot y \,|\, b\cdot (w\cdot \ep_\cG \cdot z)\ran
    =
      y^* \lan x\cdot \ep_\cG \,|\, bw \cdot \ep_\cG \ran z= \delta^{-2}y^* A(x^*bw) z\in B\cdot A(J)\cdot B.
\]
Taking spans of elements of this form, we conclude $\langle E_\cG\,|\,JE_\cG\rangle\subseteq B\cdot A(J)\cdot B.$ Conversely, given an element $yA(b)z$ for some $y,z\in B$ and $b\in J$, we have 
\[yA(b)z=\delta^2 y\langle \ep_\cG \,|\, b\cdot \ep_\cG \rangle z= \delta^2 \langle \ep_\cG \cdot y^* \, |\, b\cdot \ep_\cG \cdot z\rangle\in \langle E_\cG\,|\, JE_\cG\rangle.\] 
Therefore, $B\cdot A(J)\cdot B\subseteq \langle E_\cG\,|\, JE_\cG\rangle.$
\end{proof}
\end{lemma}


We then have the following characterization of minimality for $E_\cG$.

\begin{proposition}\label{prop:minimality_quantum_graph}
    Suppose $\cG=(B,\psi,A)$ is a quantum graph such that $E_\cG$ is full. Then $E_\cG$ is minimal if and only if there is no non-trivial ideal $J\triangleleft B$ such that $A(J)\subseteq J.$
\end{proposition}


\begin{example}[Complete quantum graphs]\label{ex:complete_graphs_simple}
The quantum edge correspondence $E_K$ for a complete quantum graph $K(B,\psi)$ is minimal: if $J\triangleleft B$ is a nonzero ideal and $A(J)\subseteq J,$ then $1\in J$ because $\delta^{-2}A(1)=\psi(1)1.$ Hence, $J=B.$ To see that $E_K$ is aperiodic, recall $E_K$ is cyclically generated as  $B-B$-bimodule inside $B\otimes_\psi B$ by $\ep_K=1\otimes 1$, so $E_K\cong B\otimes_\psi B$ as $B$-correspondences. Thus, for $n\in \bN$, we have
\[
E_K^{\otimes n}\cong (B\otimes_\psi B)\otimes_B \dots \otimes_B (B\otimes_\psi B)\cong B^{\otimes n+1}.
\]
By dimension counting, $E_K^{\otimes n}\cong B$ implies $n=0$ as long as $B\not\cong \C$. Recall $E_\cG$ is full if and only if the ideal $K=B\cdot A(B)\cdot B$ is all of $B$ (\Cref{thm:faithful_full_correspondence}). As $E_K$ is full, \Cref{thm:Schweizer} implies $\cO_{E_K}$ is simple.
\end{example}


\begin{example}[Trivial quantum graphs]\label{ex:complete_graphs_simple}
Let $(B,\psi)$ be a quantum set and let $A:B\to B$ be the identity map. The trivial quantum graph $T(B,\psi):=(B,\psi,A)$ is clearly non-minimal as $A(J)=J$ for all ideals $J$ of $B$. Furthermore, $E_T$ is cyclically generated as a $B-B$-bimodule inside $B\otimes_\psi B$ by $\ep_T=m^*(1)$, and thus is isomorphic to $B$ as $B$-correspondences, so $E_T$ is periodic. Regardless of periodicity, $E_T$ being full and non-minimal implies $\cO_{E_T}$ is non-simple.
\end{example}


Finally, we prove a statement whose contrapositive provides sufficient conditions on $\cG$ which imply aperiodicity of $E_\cG$.


\begin{proposition}\label{prop:if_periodic_then_no_sinks_and_no_sources}
Let $\cG:=(B,\psi,A)$ be a quantum graph such that $A$ is cp and $\psi$ a $\delta$-form. If $E_\cG$ is periodic, then $E_\cG$ is faithful and full ($\cG$ has no quantum sources and no quantum sinks).
\end{proposition}
\begin{proof}
Suppose $E_\cG^{\otimes n}\cong B$ for some $n\geq 1$. Since $B$ is a faithful and full as a correspondence over itself, $E_\cG^{\otimes n}$ is faithful and full. In particular, since the ideal $\langle E_\cG^{\otimes n}\,|\,E_\cG^{\otimes n}\rangle$ is a subset of $\langle E_\cG\,|\,E_\cG\rangle$, it follows that $E_\cG$ is full by \Cref{thm:faithful_full_correspondence}. 

Also by \Cref{thm:faithful_full_correspondence}, $E_\cG$ is faithful if and only if $\ker(A)$ contains no central summand of $B$. Let $p$ be a projection onto a central summand of $B$ which is contained in $\ker(A)$, so $p$ itself is central. We will show $p$ must be $0$. Let $T\colon E_\cG^{\otimes n}\to B$ be the $B$-bimodular map implementing the isomorphism $E_\cG^{\otimes n}\cong B$, and let $\xi\in E_\cG^{\otimes n}$ satisfy $T(\xi)=1$. Note $x=\langle \xi\,|\,x\cdot \xi\rangle$ for all $x\in B$. Write $\xi$ as $(\sum_{j=1}^m a_j\cdot \ep_\cG) \otimes \xi_0$ for some $a_j\in B$ and $\xi_0\in E_\cG^{\otimes n-1}$. Then
    \[
        p 
        = \<\xi \,|\, p\cdot \xi\> 
        = \sum_{j,k=1}^m \< a_j\cdot \ep_\cG\otimes \xi_0\, |\, p a_k\cdot\ep_\cG \otimes \xi_0\> 
        = \sum_{j,k=1}^m \delta^{-2}\< \xi_0\, \lvert \, A(a_j^* p a_k)\cdot \xi_0\>.
    \]
Since $a_j^* p a_k$ is contained in $\ker(A)$ for each $1\leq j,k\leq m$, we have $p=0$.
\end{proof}


\bibliographystyle{amsalpha}
\bibliography{final_arxiv_1.bib}
\end{document}